\documentclass[11pt,leqno]{amsart}

\usepackage[top=2.5 cm,bottom=2 cm,left=2.5 cm,right=2 cm]{geometry}
\usepackage[utf8]{inputenc}
  \usepackage{color}

\usepackage{esint}
\usepackage{graphicx}
\usepackage{hyperref}

\newtheorem{example}{Example}[section]
\newtheorem{definition}{Definition}[section]
\newtheorem{theorem}{Theorem}[section]

\newtheorem{lemma}{Lemma}[section]
 \newtheorem{corollary}{Corollary}[section]

\newtheorem*{maintheorem*}{Main Theorem}
\allowdisplaybreaks
\numberwithin{equation}{section}

\renewcommand{\i}{\ifmmode\mathit{\mathchar"7010 }\else\char"10 \fi}
\renewcommand{\j}{\ifmmode\mathit{\mathchar"7011 }\else\char"11 \fi}
\newcommand{\R}{\mathbb{R}}
\newcommand{\N}{\mathbb{N}}

\newcommand{\norm}[1]{\left\|#1\right\|}

\newcommand{\weakstar}{\overset{\star}\rightharpoonup}
\newcommand{\weak}{\rightharpoonup}

\newcommand{\pt}{\partial_t}
\newcommand{\ps}{\partial_s}
\newcommand{\pttt}{\partial_{ttt}^3}
\newcommand{\ptt}{\partial_{tt}^2}
\newcommand{\px}{\partial_x }
\newcommand{\pxx}{\partial_{xx}^2}
\newcommand{\pxxx}{\partial_{xxx}^3}
\newcommand{\pxxxx}{\partial_{xxxx}^4}

\newcommand{\ptxxxx}{\partial_{txxxx}^4}

\newcommand{\ptxx}{\partial_{txx}^3}

\newcommand{\vfi}{\varphi}
\newcommand{\eps}{\varepsilon}

\def\begi{\begin{itemize}}
\def\endi{\end{itemize}}
\def\bega{\begin{array}}
\def\enda{\end{array}}

\def\bel{\begin{equation}\label}
\def\eeq{\end{equation}}

\newenvironment{Assumptions}
{%

\begin{enumerate}}%
{\end{enumerate}}

{%

\begin{enumerate}}%
{\end{enumerate}}

\begin{document}

\title[Waves in flexural beams with nonlinear adhesive interaction]{Waves in flexural beams \\with nonlinear adhesive interaction}

\author{G. M. Coclite}
\author{G. Devillanova}
\author{F. Maddalena}
\address[Giuseppe Maria Coclite, Giuseppe Devillanova, and Francesco Maddalena]{\newline
  Dipartimento di Meccanica, Matematica e Management, Politecnico di Bari,
  Via E.~Orabona 4,I--70125 Bari, Italy.}
\email[]{giuseppemaria.coclite@poliba.it}
\email[]{giuseppe.devillanova@poliba.it}
\email[]{francesco.maddalena@poliba.it}

\date{\today}



\thanks{The authors are members of the Gruppo Nazionale per l'Analisi Matematica, la Probabilit\`a e le loro Applicazioni (GNAMPA) of the Istituto Nazionale di Alta Matematica (INdAM). The second author is also supported by MIUR - FFABR - 2017 research grant. The authors have been partially supported by the Research Project of National Relevance ``Multiscale Innovative Materials and Structures'' granted by the Italian Ministry of Education, University and Research (MIUR Prin 2017, project code 2017J4EAYB and the Italian
Ministry of Education, University and Research under the Programme Department of Excellence Legge 232/2016 (Grant No. CUP - D94I18000260001)}

\begin{abstract} 
The paper studies the  initial boundary value problem related to the dynamic evolution of an elastic beam interacting with 
a substrate through an elastic-breakable forcing term. This discontinuous interaction is aimed to model the phenomenon of 
attachement-detachement of the beam occurring in adhesion phenomena. We prove existence of solutions in energy space and 
exhibit various counterexamples to uniqueness.
Furthermore we characterize some relavant   features of the solutions, ruling the main effectes of the nonlinearity due to the elasic-breakable term on the dynamical evolution, 
by proving the linearization property according to \cite{G96} and an asymtotic result pertaining the long time behavior.
\end{abstract}

\maketitle


\section{Introduction}
\label{sec:intro}

In a broader sense the term {\it adhesion} refers to a physical situation in which two material bodies, during their mechanical evolution,  experience a contact interaction which fails in (possibly bounded) regions of space-time and restores after a while. The phenomenon strongly depends on the constitutive properties of the involved materials and a crucial problem relies in understanding the nature of the interaction. The manifestation of such such phenomenon occurs at every scale,  ranging from DNA molecules to  the structural engineering works. Mathematics has looked at these problems, at least in the stationary case,  since 
 the seminal paper  \cite{BK}  and 
in some recent works (see, e.g., \cite{MP, MPPT, MPT, MPT2, MPT21}) one of the authors contributed to the  study of  the static problem  of adhesion of elastic  structures by exploiting  different constitutive assumptions to the aim of 
 characterizing, in a variational framework,  the interplay of the  of debonding with other constitutive properties. 
 
 However the dynamical problem is a different story since at the heart of the question stays the understanding of the 
 attachement-detachement occurrence and how this affects  the whole evolution problem. This is a subtle problem since the analytical tools at disposal, such as spatio-temporal estimates in some norms, seem too rough to catch exhaustive quantitative informations, even in short time. 
 
 In \cite{CFLM, CFLM-1}  the simplest mechanical model consisting in elastic string was considered, whereas a discontinuous forcing term was assumed to model the adhesive interaction of the string with a rigid substrate. The resulting mathematical problem is then ruled by an initial boundary value problem for a  semilinear  second order wave  equation and the results in  \cite{CFLM,  CFLM-1}  show   some    tricky peculiarities of the problem. As it is well known in continuum mechanics, the other basic model for  one dimensional elastic structures
 is represented by the so called  Bernoulli-Navier beam  governing the flexural deformations of a slender material body. In the linear elastic framework the equation expressing the balance of momentum is ruled by a fourth order spatial differential operator. 
 Analogously to \cite{CFLM}, we assume a discontinuous forcing term to model the adhesive interaction of the beam with the external environment. One can visualize as a physical situation an elastic beam connected to a rigid  substrate through a
foundation made of continuous distribution of elastic-breakable springs.\\
 We study the well posedness of the mathematical problem  proving existence of global in time solutions 
  in the natural  energy space and exploit some features of the solutions to obtain information about the role played by the attachment-detachment occurrence on the dynamical evolution. Indeed, in \cite{CFM} it was proved  
 that the main
effects induced by the nonlinearity at the transition from attached to detached states consist in 
a loss of regularity of the solution and in a migration of the total energy through the scales. 
Here we deepen the analysis by using the {\it linearization condition} introduced by Patrick {G\'erard}  in \cite{G96} according to which one can conclude that if  the semilinear  evolution problem satisfies such condition,  then the nonlinear forcing term {\it does not  induce any new oscillation or energy concentrations} (\cite{G96} ).  Furthermore we prove an asymptotic result 
for the long time behavior in the case of bounded solutions, asserting  the occurrence of three mutually exclusive states: the trivial one, the totally detached state, the totally attached state.\\
The paper is organized as follows. In Section \ref{sec:pb} we formulate the initial boundary value problem.  In Section \ref{sec:main} we state the main results of the paper consisting in 
Theorem \ref{th:exist} (Existence of solutions),  Theorem \ref{th:adhesion} (Characerization of adhesive states),
Theorem \ref{th:asymp} (Long time behavior), Theorem \ref{th:lin} (Linearization property).
The proofs of these theorems are given respectively in Sections \ref{sec:exist}, \ref{sec:adhesion},
\ref{sec:asymp},  \ref{sec:lin}.
In Section \ref{sec:uniqu} we provide some examples showing non-uniqueness and lack ofsmoothness of the solutions.

\section{Statement of the problem}
\label{sec:pb}

Let us consider an elastic beam under Bernoulli-Navier constitutive assumption, occupying in the reference configuration the 
interval $[0,L]\subset\R$, the balance of linear momentum  delivers 
 the semilinear initial boundary value problem 
\begin{equation}
\label{eq:el}
\begin{cases}
\rho\ptt u=-\mu\pxxxx u-\Phi'\left(u\right)&\quad t>0,0<x<L,\\
\pxx u(t,0)=\pxx u(t,L)=0&\quad t>0,\\
\pxxx u(t,0)=\pxxx u(t,L)=0&\quad t>0,\\
u(0,x)=u_0(x)&\quad 0<x<L,\\
\pt u(0,x)=u_1(x)&\quad 0<x<L.
\end{cases}
\end{equation}
We shall assume that
\begin{Assumptions}
\item \label{ass:phi} $\Phi\in C(\R)\cap C^2(\R\setminus\{1,-1\})$, $\Phi$ is positive, constant in $(-\infty,-1]$ and in $[1,\infty)$, 
convex in $[-1,1]$, decreasing in $[-1,0]$, increasing in $[0,1]$, and $\Phi(u)\ge \kappa u^2$ in $[-1,1]$ for some constant $\kappa>0$;
\item \label{ass:init} $u_0\in H^2(0,L)$, $u_1\in L^2(0,L)$;
\item \label{ass:cost} $\rho>0$ is the constant material density and $\mu>0$ is the elastic stiffness.
\end{Assumptions}

As a consequence of \ref{ass:phi}, $\Phi'$ has a jump discontinuity in $u=\pm1$ and
\begin{align*}
u\in(-\infty,-1)\cup(1,\infty)&\Rightarrow \Phi'(u)=0,\\
0<u<1&\Rightarrow 0<\Phi'(u)\le \lim_{u\to1^-}\Phi'(u),\\
-1<u<0&\Rightarrow 0>\Phi'(u)\ge \lim_{u\to-1^+}\Phi'(u).
\end{align*}

To fix ideas, a function satisfying such assumption is
\begin{equation}
\label{eq:Phi}
\Phi(u)=\begin{cases}
u^2 &\qquad \text{if $|u|\le 1$},\\
1 &\qquad \text{if $|u|> 1$}.
\end{cases}
\end{equation}
In particular we have for all $u\neq \pm 1$
\begin{equation}
\label{eq:Phi'}
\Phi'(u)=\begin{cases}
2u &\qquad \text{if $|u|< 1$},\\
0 &\qquad \text{if $|u|>1$}.
\end{cases}
\end{equation}

The forcing term $\Phi'(u)$ is thought to model the elastic breakable interaction between the beam and the external environment. It confers to the problem a localized nonlinearity which affects the evolution in a significant way.\\
The natural energy associated to the Problem \eqref{eq:el} (i.e.\ to any solution $u$ to \eqref{eq:el}), is given at time $t$ by the quantity
\begin{equation}
\label{en}
E[u](t)=\int_0^L\left(\frac{\rho(\pt u(t,x))^2+\mu(\pxx u(t,x))^2}{2}+\Phi(u(t,x))\right)dx.
\end{equation}
In general,  the lack of Lipschitz continuity in the nonlinear term $\Phi'$ suggests we cannot expect the existence of conservative solutions, i.e.\ solutions that 
preserve energy. 
Also  the physics underlying  the problem, foresees a kind of {\it dissipation} when the material is detaching from the substrate and this is 
accompanied by hysteresis cycles (see e.g.  {\cite[Appendix B]{MPPT}}).

\section{Main results}
\label{sec:main}

We begin by studying  the well-posedness of Problem \eqref{eq:el} and  the  regularity of its solutions (see Definition \ref{def:sol} below).
We prove the existence of Lipshitz continuous dissipative solutions and give examples of distinct solutions to \eqref{eq:el} which do not depend continuously on the initial conditions.

\begin{definition}
\label{def:sol}
We say that a function $u:[0,\infty)\times[0,L]\to\R$ is a dissipative solution of \eqref{eq:el} if
\begin{itemize}
\item[($i$)] $u\in C([0,\infty)\times[0,L])$;
\item[($ii$)] $\pt u,\,\pxx u\in L^\infty(0,\infty;L^2(0,L))$;
\item[($iii$)] $u$ is a \em{weak solution} to \eqref{eq:el}, i.e.\ for every test function $\vfi\in C^\infty(\R^2)$ with compact support
\begin{equation}
\label{eq:weak}
\begin{split}
\int_0^\infty\int_0^L& \left(\rho u\ptt \vfi+\mu\pxx u\pxx \vfi+h_u\vfi\right)dtdx\\
&-\int_0^L \rho u_1(x)\vfi(0,x)dx+\int_0^L \rho u_0(x)\pt\vfi(0,x)dx=0,
\end{split}
\end{equation}
where $h_u\in\partial \Phi'\left(u\right)$, that is the subdifferential of $\Phi'(u)$;
\item[($iv$)] $u$ may \em{dissipate energy}, i.e.\ for almost every $t>0$: $E[u](t)\leq E[u](0)$, namely (see \eqref{en})
\begin{equation}
\label{eq:energydissip}
\begin{split}
E[u](t)=\int_0^L&\left(\frac{\rho(\pt u(t,x))^2+\mu(\pxx u(t,x))^2}{2}+\Phi(u(t,x))\right)dx\\
&\qquad\qquad\le \int_0^L\left(\frac{\rho(u_1(x))^2+\mu(\pxx u_0(x))^2}{2}+\Phi(u_0(x))\right)dx=E[u](0).
\end{split}
\end{equation}
\end{itemize}
\end{definition}
We remind that
\begin{equation*}
h\in\partial \Phi'\left(u\right)
\end{equation*}
means that $h:[0,\infty)\times [0,L]\to\R$ sastifies
\begin{equation*}
h_u(t,x)\begin{cases}
=\Phi'(u(t,x)),&\text{if $|u(t,x)|<1$,}\\
=0,&\text{if $|u(t,x)|>1$,}\\
\in [0,\Phi'(1^-)],&\text{if $u(t,x)=1$,}\\
\in [\Phi'(-1^+),0],&\text{if $u(t,x)=-1$.}
\end{cases}
\end{equation*}
\vspace{11pt}

Let us state the following theorem  asserting the existence of dissipative solutions.

\begin{theorem}[{\bf Existence}]
\label{th:exist}
If $\Phi$ satisfies  \ref{ass:phi} and $u_0$, $u_1$ satisfy \ref{ass:init}, 
then problem \eqref{eq:el} admits a  dissipative  solution in the sense of 
Definition \ref{def:sol}.
\end{theorem}
\vspace{11pt}
The following  result provides a sufficient condition ruling the non-detachment of dissipative solutions in dependence on the initial data.

\begin{theorem}
\label{th:adhesion}
Let $u$ be a dissipative solution of \eqref{eq:el}.
If
\begin{equation}
\label{eq:ad1}
\norm{u_0}_{L^\infty(0,L)}<1\qquad\mbox{ and } \qquad E[u](0)<\frac{4\kappa\mu}{3\vee 2\kappa}(\le2\mu),
\end{equation}
where $\kappa$ is defined in \ref{ass:phi}, then
\begin{equation}
\label{eq:ad1*}
\norm{u}_{L^\infty((0,\infty)\times(0,L))}<1.
\end{equation}
\end{theorem}
\vspace{11pt}

 The long time behavior is a very subtle problem for evolutionary partial differential equations, so by restricting the focus on bounded  dissipative solutions, we are able to prove the following statement.

\begin{theorem}[{\bf Long Time Behavior}]
\label{th:asymp}
Let  $u$ be a dissipative solution of \eqref{eq:el} and $\{t_n\}_{n\in\N}\subset(0,\infty)$ such that $t_n\to\infty$.
If 
\begin{equation}
\label{eq:bound}
u\in L^\infty((0,\infty)\times(0,L)),
\end{equation}
then there exist a subsequence $\{t_{n_k}\}_{k\in\N}$ and two  constants $a,\, b\in\R$ such that
\begin{equation*}
u(t_{n_k},\cdot)\weak u_\infty \quad \text{weakly in $H^2(0,L)$ as $k\to\infty$,}
\end{equation*}
where
\begin{equation}
\label{eq:aff}
u_\infty(x)=ax+b.
\end{equation}
Moreover, only one of the following statements can occur
\begin{align}
\label{eq:s1}
& \text{$u_\infty\equiv 0$,}\\
\label{eq:s2}
& \text{$u_\infty(x)\ge1$ for every $x\in[0,L]$,}\\
\label{eq:s3}
& \text{$u_\infty(x)\le -1$ for every $x\in[0,L]$.}
\end{align}
\end{theorem}
\vspace{11pt}

The question at the basis of the last subsequent  result can be formulated as follows: 
{\it How the nonlinearity characterizing  the forcing term $\Phi'$ affects the evolution problem}?
We retain that {G\'erard}'s linearization condition provides a precise mathematical tool to answer to the previous rather vague question. Indeed the absence of further  energy concentrations or oscillations due to the nonlinearity, suspected to arise in correspondence of the attachment-detacment process, constitutes an interesting property in itself, also considering the nonuniqueness
of the  solutions  $\{u_n\}_n$ below.

\begin{theorem}[{\bf Linearization Property}]
\label{th:lin}
Let $\{u_{0,n}\}_n\subset H^2(0,L),\,\{u_{1,n}\}_n\subset L^2(0,L)$, $\{u_n\}_n$ be a sequence of dissipative solutions of \eqref{eq:el} in correspondence of 
such initial data, and $v_n$ be the dissipative solution of the linearized problem
\begin{equation}
\label{eq:v}
\begin{cases}
\rho\ptt v_n=-\mu\pxxxx v_n&\quad t>0,0<x<L,\\
\pxx v_n(t,0)=\pxx v_n(t,L)=0&\quad t>0,\\
\pxxx v_n(t,0)=\pxxx v_n(t,L)=0&\quad t>0,\\
v_n(0,x)=u_{0,n}(x)&\quad 0<x<L,\\
\pt v_n(0,x)=u_{1,n}(x)&\quad 0<x<L.
\end{cases}
\end{equation}
If
\begin{align}
\label{ass:lin1}
  &\text{$u_{0,n}\weak 0$ weakly in $H^2(0,L)$},\\
\label{ass:lin2}
  &\text{$u_{1,n}\weak 0$ weakly in $L^2(0,L)$},\\
\label{ass:lin3}
&\limsup_n\int_0^L\left(\frac{\rho u_{1,n}^2+\mu (\pxx u_{0,n})^2}{2}\right)dx<\frac{4\kappa\mu}{3\vee 2\kappa}(\le2\mu),
\end{align}
where $k$ is defined in \ref{ass:phi}, then the following linearization condition holds true
\begin{equation}
\label{eq:claimlin}
\norm{\pt (u_n- v_n)}_{L^\infty(0,T;L^2(0,L))}+\norm{\pxx (u_n- v_n)}_{L^\infty(0,T;L^2(0,L))}\to0
\end{equation}
for every $T\ge0$ as $n\to \infty$.
\end{theorem}

\section{Existence of Dissipative Solutions}
\label{sec:exist}

This section is dedicated to the proof of Theorem \ref{th:exist}.

Our argument is based on the approximation of the Neumann problem \eqref{eq:el}
with a sequence of Neumann problems  \eqref{eq:eln}   characterized by  smooth source terms and smooth initial data.
More precisely, 
let $\{u_{0,n}\}_{n\in\N},\,\{u_{1,n}\}_{n\in\N}\subset C^\infty([0,L]),\,\{\Phi_n\}_{n\in\N}\subset C^\infty(\R)$, 
for every $n\in\N$ consider the approximating problems

\begin{equation}
\label{eq:eln}
\begin{cases}
\rho\ptt u_n=-\mu\pxxxx u_n-\Phi_n'(u_n) &\quad t>0,0<x<L,\\
\pxx u_n(t,0)=\pxx u_n(t,L)=0 &\quad t>0,\\
\pxxx u_n(t,0)=\pxxx u_n(t,L)=0 &\quad t>0,\\
u_n(0,x)=u_{0,n}(x)&\quad 0<x<L,\\
\pt u_n(0,x)=u_{1,n}(x)&\quad 0<x<L,
\end{cases}
\end{equation}
where $\{u_{0,n}\}_{n\in\N},\,\{u_{1,n}\}_{n\in\N},$ $\{\Phi_n\}_{n\in\N}$ are sequences of smooth approximations of 
$u_0,\,u_1$, and $\Phi$ respectively, i.e. they satisfy the following requirements
\begin{equation}
\label{eq:assn}
\begin{split}
&u_{0,n}\to u_0\quad\text{in $H^1(0,L)$},\quad u_{1,n}\to u_1\quad\text{in $L^2(0,L)$},\quad \Phi_n\to \Phi \quad \text{uniformly in $\R$},\\
&\Phi_n'\to \Phi' \quad \text{pointwise  and uniformly in 
$\R\setminus\left((-1-\eps,-1+\eps)\cup (1-\eps,1+\eps)\right)$ for every $\eps$},\\
&|u|\ge1+\eps\Rightarrow \Phi_n'(u)=0,\qquad \eps>0,\> n\in\N,\\
&\norm{u_{0,n}}_{H^1(0,L)}\le C,\quad \norm{u_{1,n}}_{L^2(0,L)}\le C, \quad 0\le \Phi_n,\,\Phi_n'\le C,\qquad n\in\N,\\
&u_{0,n}^{\prime\prime}(0)=u_{0,n}^{\prime\prime}(L)=u_{1,n}(0)=u_{1,n}(L)=0,\qquad n\in\N,
\end{split}
\end{equation}
where $C$ is a positive constant which does not depend on $n$.

For any $n\in \N$, \eqref{eq:eln} admits a classical solution for short time thanks to the Cauchy-Kowaleskaya Theorem (see \cite{taylor}). Furthermore, 
for such a problem, solutions are indeed global in time thanks to the following results.
Let $u_n$ be the unique classical solution to \eqref{eq:eln}.

\begin{lemma}[{\bf Energy conservation}]
\label{lm:energy}
Classical solution to \eqref{eq:eln} preserves energy.
\end{lemma}

\begin{proof}
Set $E_n:=E[u_n]$  the energy corresponding to $u_n$.
We have to prove that the function $t\mapsto E_n(t)$ is constant (with constant value $E_n(0)$). Indeed, we have that
\begin{align*}
E_n'(t)=&\frac{d}{dt}\int_0^L\left(\frac{\rho(\pt u_n)^2+\mu(\pxx u_n)^2}{2}+\Phi_n(u_n)\right)dx\\
=&\int_0^L\left(\rho\pt u_n\ptt u_n+\mu\pxx u_n\ptxx u_n+\Phi_n'(u_n)\pt u_n\right)dx\\
=&\int_0^L\pt u_n\underbrace{\left(\rho\ptt u_n+\mu\pxxxx u_n+\Phi_n'(u_n)\right)}_{=0}dx=0.
\end{align*}
\end{proof}

As a consequence of energy conservation, since the functions $\Phi_n$ are positive, 
we have the following boundedness result.
\begin{corollary}
\label{co:energy}
The sequences
$\{\pt u_n\}_{n\in\N}$ and $\{\pxx u_n\}_{n\in\N}$ are bounded in
$L^\infty(0,\infty;L^2(0,L))$.
\end{corollary}

\begin{lemma}[{\bf $L^2$ estimate}]
\label{lm:l2}
The sequence $\{u_n\}_{n\in\N}$ is bounded in $L^\infty(0,T;L^2(0,L))$, for every $T>0$.
\end{lemma}

\begin{proof}
Using the H\"older inequality
\begin{align*}
\int_0^L u_n^2(t,x)dx=&\int_0^L\left(u_{0,n}(x)+\int_0^t \ps u_n(s,x)ds\right)^2dx\\
\le&2\int_0^L u_{0,n}^2(x)dx+2 \int_0^L\left(\int_0^t |\ps u_n(s,x)|ds\right)^2dx\\
\le&2\int_0^L u_{0,n}^2(x)dx+2t \int_0^t\int_0^L (\ps u_n(s,x))^2dxds\\
\le&2\int_0^L u_{0,n}^2(x)dx+2t^2 \sup_{s\ge0}\int_0^L (\ps u_n(s,x))^2dx ,
\end{align*}
the claim follows from Corollary \ref{co:energy}.
\end{proof}

The following result follows from Corollary \ref{co:energy} by a straightforward application of Gagliardo-Nirenberg Interpolation Inequality (see for instance \cite[Theorem at page 125]{N}).

\begin{lemma}[{\bf $H^1$ estimate}]
\label{pr:h1}
The sequence 
$\{\px u_n\}_{n\in\N}$ is bounded in $L^\infty(0,T;L^2(0,L))$, for every $T>0$.
\end{lemma}


\begin{lemma}[{\bf $L^\infty$ estimate}]
\label{lm:linfty}
The sequence $\{u_n\}_{n\in\N}$ is bounded in $L^\infty((0,T)\times(0,L))$, for every $T>0$.
\end{lemma}

\begin{proof}
Fix $0<t<T$.
Lemmas \ref{lm:energy} and \ref{lm:l2} and Corollary \ref{co:energy} imply that $\{u_n\}_{n\in\N}$ is bounded in $L^\infty(0,T;H^1(0,L))$. 
Since $H^1(0,L)\subset L^\infty(0,L)$ we have
\begin{equation*}
|u_n(t,x)|\le \norm{u_n(t,\cdot)}_{L^\infty(0,L)}\le c\norm{u_n(t,\cdot)}_{H^1(0,L)}
\le c\norm{u_n}_{L^\infty(0,T;H^1(0,L))},\quad (t,x)\in (0,T)\times(0,L),
\end{equation*}
for some constant $c>0$ depending only on $L$. Therefore 
\begin{equation*}
\norm{u_n}_{L^\infty((0,T)\times(0,L))}\le c\norm{u_n}_{L^\infty(0,T;H^1(0,L))},
\end{equation*}
that gives the claim.
\end{proof}

\begin{lemma}[{\bf space Lipschitz estimate}]
\label{lm:lip}
The sequence 
$\{\px u_n\}_{n\in\N}$ is bounded in $L^\infty((0,T)\times(0,L))$, for every $T>0$.
\end{lemma}

\begin{proof}
Fix $0<t<T$ and $0<x<L$.
Lemmas \ref{lm:energy}, \ref{lm:l2}, and \ref{pr:h1} imply that $\{u_n\}_{n\in\N}$ is bounded in $L^\infty(0,T;H^2(0,L))$. 
Since $H^1(0,L)\subset L^\infty(0,L)$ we have, for  every $(t,x)\in (0,T)\times(0,L)$,
\begin{align*}
|\px u_n(t,x)|&\le \norm{\px u_n(t,\cdot)}_{L^\infty(0,L)}\\
&\le c\norm{\px u_n(t,\cdot)}_{H^1(0,L)}\le c\norm{u_n(t,\cdot)}_{H^2(0,L)}\\
&\le c\norm{u_n}_{L^\infty(0,T;H^2(0,L))},
\end{align*}
for some constant $c>0$ depending only on $L$. Therefore 
\begin{equation*}
\norm{\px u_n}_{L^\infty((0,T)\times(0,L))}\le c\norm{u_n}_{L^\infty(0,T;H^2(0,L))},
\end{equation*}
that gives the claim.
\end{proof}

\begin{proof}[Proof of Theorem \ref{th:exist}]
Thanks to Lemmas \ref{lm:energy}, \ref{lm:l2} and \cite[Theorem 5]{S}
 there exists a function $u$ satisfying items ($i$) and ($ii$) in Definition \ref{def:sol}  and a function $h_u\in L^\infty ((0,T)\times(0,L)),\, h_u\in \partial \Phi'(u),$ such that,
passing to a subsequence,
\begin{equation}
\label{eq:conv}
\begin{split}
 u_n\weak u &\quad \text{in $H^1((0,T)\times(0,L))$ and in $L^2(0,T;H^2(0,L))$, for each $T\ge0$},\\ 
 u_n \to u &\quad \text{in $L^\infty((0,T)\times(0,L))$, for each $T\ge0$},\\
 \Phi_n'(u_n) \weak h_u &\quad \text{in $L^p((0,T)\times(0,L))$, for each $T\ge0$ and $1\le p<\infty$}.
\end{split}
\end{equation}

We have yet to verify that $u$ is a weak solution of \eqref{eq:el} i.e. Definition \ref{def:sol} - item ($iii$).
Let $\vfi\in C^\infty(\R^2)$ be a test function with compact support, since $u_n$ is a solution to \eqref{eq:eln}, we have that for every $n$
\begin{align*}
\int_0^\infty\int_0^L& \left(\rho u_n\ptt \vfi+\mu \pxx u_n\pxx \vfi+\Phi_n'(u_n)\vfi\right)dxdt\\
&-\int_0^L \rho u_{1,n}(x)\vfi(0,x)dx+\int_0^L \rho u_{0,n}(x)\pt\vfi(0,x)dx=0.
\end{align*}
Then, by taking the limit as $n\to\infty$,
\eqref{eq:weak} follows by using \eqref{eq:assn} and \eqref{eq:conv}.

Finally, due to \eqref{eq:conv} and \eqref{eq:assn} we have
 \begin{align*}
 \pt u_n\weak \pt u &\quad \text{in $L^p(0,T;L^2(0,L))$ for each $T\ge0$ and $1\le p<\infty$},\\ 
 \pxx u_n\weak \pxx u &\quad \text{in $L^p(0,T;L^2(0,L))$ for each $T\ge0$ and $1\le p<\infty$},\\ 
 \Phi_n(u_n) \to \Phi(u) &\quad \text{in $L^\infty((0,T)\times(0,L))$, for each $T\ge0$}.
\end{align*}
Therefore, Definition \ref{def:sol} - item ($iv$) follows by the lower semicontinuity of the $L^2$ norm with respect to the weak convergence by taking into account Lemma \ref{lm:energy}.
\end{proof}

\section{Adhesive states}
\label{sec:adhesion}

This section is dedicated to the proof of Theorem  \ref{th:adhesion}.

\begin{proof}[Proof of Theorem \ref{th:adhesion}]
Since $u$ is continuous by \eqref{eq:ad1} there exists $\tau>0$ such that in the short time interval 
$[0,\tau]$ we have
\begin{equation*}
|u(t,x)|<1,\qquad (t,x)\in [0,\tau]\times[0,L].
\end{equation*}
So we can define $\tau^*$ as follows
\begin{equation*}
\tau^*=\sup\left\{\tau>0 \;|\; |u(t,x)|<1 \mbox{ for all }(t,x)\in [0,\tau]\times[0,L]\right\}.
\end{equation*}
We claim that 
\begin{equation}
\label{eq:ad2}
\tau^*=\infty.
\end{equation}
Observe that, due to \ref{ass:phi},  
\begin{equation*}
\Phi(u(t,x))\ge \kappa u^2(t,x),\qquad (t,x)\in [0,\tau^*)\times[0,L].
\end{equation*}
Therefore
\begin{equation*}
E[u](t)\ge \int_0^L\left(\frac{\rho(\pt u(t,x))^2+\mu(\pxx u(t,x))^2}{2}+\kappa(u(t,x))^2\right)dx,\qquad t\in [0,\tau^*),
\end{equation*}
and, in particular, 
\begin{equation*}
E[u](t)\ge \frac{\mu\norm{\pxx u(t,\cdot)}_{L^2(0,L)}^2}{2}+\kappa\norm{u(t,\cdot)}_{L^2(0,L)}^2,\qquad t\in [0,\tau^*).
\end{equation*}
Since $u$ is dissipative, using the Sobolev embedding $H^1(0,L)\subset L^\infty(0,L)$ (see \cite[Theorem 8.5]{LL}) and Lemma \ref{pr:h1}, we have from every $t\in [0,\tau^*)$
\begin{align*}
\norm{u(t,\cdot)}_{L^\infty(0,L)}^2\le& \frac{\norm{u(t,\cdot)}_{H^1(0,L)}^2}{2}=\frac{\norm{u(t,\cdot)}_{L^2(0,L)}^2+\norm{\px u(t,\cdot)}_{L^2(0,L)}^2}{2}\\
\le&\frac{3}{4}\norm{u(t,\cdot)}_{L^2(0,L)}^2+\frac{\norm{\pxx u(t,\cdot)}_{L^2(0,L)}^2}{4}\\
\le&\left(\frac{3}{4\kappa}\vee\frac{1}{2}\right)\left(\frac{\norm{\pxx u(t,\cdot)}_{L^2(0,L)}^2}{2}+\kappa\norm{u(t,\cdot)}_{L^2(0,L)}^2\right)\\
\le&\frac{1}{\mu}\left(\frac{3}{4\kappa}\vee\frac{1}{2}\right) E(t)\le\frac{1}{\mu}\left(\frac{3}{4\kappa}\vee\frac{1}{2}\right) E(0)<1,
\end{align*}
(where the last inequality holds thanks to \eqref{eq:ad1})
that proves \eqref{eq:ad2}.
\end{proof}

\section{Long Time Behavior}
\label{sec:asymp}

This section is dedicated to the proof of Theorem  \ref{th:asymp}.

Let $u$  be a dissipative solution of \eqref{eq:el} satisfying \eqref{eq:bound}.

\noindent {\bf STEP 1.} We begin by deducing the effective asymptotic problem.

Consider the functions
\begin{equation*}
u_\tau(t,x)=u(\tau t, x),\qquad \tau>0,\>t\ge0,\>x\in[0,L].
\end{equation*}
$u_\tau$ is a dissipative solution of  the initial boundary value problem
\begin{equation}
\label{eq:eltau}
\begin{cases}
\displaystyle\frac{\rho\ptt u_\tau}{\tau^2}=-\mu\pxxxx u_\tau-\Phi'\left(u_\tau\right)&\quad t>0,0<x<L,\\
\pxx u_\tau(t,0)=\pxx u_\tau(t,L)=0&\quad t>0,\\
\pxxx u_\tau(t,0)=\pxxx u_\tau(t,L)=0&\quad t>0,\\
u_\tau(0,x)=u_0(x)&\quad 0<x<L,\\
\pt u_\tau(0,x)=\tau u_1(x)&\quad 0<x<L,
\end{cases}
\end{equation}
in the sense of Definition \ref{def:sol}, namely
\begin{itemize}
\item[($iii$)]  for every test function $\vfi\in C^\infty(\R^2)$ with compact support
\begin{equation}
\label{eq:weaktau}
\begin{split}
\int_0^\infty\int_0^L& \left(-\frac{\rho\pt u_\tau}{\tau^2}\pt \vfi+ \mu\pxx u_\tau\pxx \vfi+h_{\tau}\vfi\right)dtdx-\int_0^L \frac{\rho u_1(x)}{\tau}\vfi(0,x)dx=0,
\end{split}
\end{equation}
where $h_{\tau}\in\partial \Phi'\left(u_\tau\right)$, that is the subdifferential of $\Phi'(u_\tau)$;
\item[($iv$)] $u_\tau$ may \em{dissipate energy}, i.e.\ for almost every $t>0$: 
\begin{equation}
\label{eq:energydissiptau}
\begin{split}
\int_0^L&\left(\frac{\rho(\pt u_\tau(t,x))^2}{2\tau^2}+\frac{\mu(\pxx u_\tau(t,x))^2}{2}+\Phi(u_\tau(t,x))\right)dx\\
&\qquad\le \int_0^L\left(\frac{\rho(u_1(x))^2}{2}+\frac{\mu(\pxx u_0(x))^2}{2}+\Phi(u_0(x))\right)dx.
\end{split}
\end{equation}
\end{itemize}

Thanks to \ref{ass:phi}, \eqref{eq:bound}, and \eqref{eq:energydissiptau},
\begin{align*}
&\text{$\{ u_\tau\}_{\tau>0}$ is bounded in $L^\infty(0,\infty;H^2(0,L))$},\\
&\text{$\{h_{\tau}\}_{\tau>0}$ is bounded in $L^\infty((0,\infty)\times(0,L))$},
\end{align*}
there exists two functions $U\in L^\infty(0,\infty;H^2(0,L)),\, H\in L^\infty((0,\infty)\times(0,L))$ such that, passing to a subsequence,
\begin{align*}
u_\tau\weakstar U\qquad &\text{weakly-$\star$  in $L^\infty_{loc}((0,\infty)\times(0,L))$ as $\tau\to\infty$},\\
h_\tau\weakstar H\qquad &\text{weakly-$\star$  in $L^\infty_{loc}((0,\infty)\times(0,L))$ as $\tau\to\infty$}.
\end{align*}
Using \eqref{eq:energydissiptau}
\begin{align*}
&\text{$\{\pt u_\tau/\tau\}_{n\in\N}$ is bounded in $L^\infty(0,\infty;L^2(0,L))$},
\end{align*}
therefore sending $\tau\to\infty$ in \eqref{eq:weaktau} we get 
\begin{equation}
\label{eq:weaktaulim}
\int_0^\infty\int_0^L \left(\mu\pxx U\pxx \vfi+H\vfi\right)dtdx=0,
\end{equation}
namely $U=U(x)$, $H=H(x)$ and the effective asymptotic problem is
\begin{equation}
\label{eq:eltauinfty}
\begin{cases}
-\mu\pxxxx U=H,&\quad 0<x<L,\\
\pxx U(0)= \pxx U(L)=0,&{}\\
\pxxx U(0)=\pxxx U(L)=0.&{}
\end{cases}
\end{equation}

\noindent {\bf STEP 2.} 
We exploit   more subtle characterizations of the limit functions $U$ and $H$.
To this aim we fix a sequence $\{t_n\}_{n\in\N}\subset(0,\infty)$ such that $t_n\to\infty$ and study the convergence of the sequence
\begin{equation*}
\{u(t_n,\cdot)\}_{n\in\N}.
\end{equation*}
Since we have the dissipation inequality \eqref{eq:energydissip} and the assumption \eqref{eq:bound}, we gain
\begin{align*}
&\text{$\{u(t_n,\cdot)\}_{n\in\N}$ is bounded in $H^2(0,L)$},\\
&\text{$\{h_u(t_n,\cdot)\}_{n\in\N}$ is bounded in $L^\infty(0,L)$}.
\end{align*}
Therefore there exist two functions $u_\infty\in H^2(0,L),\, h_\infty\in L^\infty(0,L)$ such that passing to a subsequence 
\begin{equation}
\label{eq:conv1}
\begin{split}
&u(t_n,\cdot)\weak u_\infty\qquad \text{weakly in $H^2(0,L)$ as $n\to\infty$},\\
&u(t_n,\cdot)\to u_\infty\qquad \text{a.e. in $(0,L)$ as $n\to\infty$},\\
&h_u(t_n,\cdot)\weakstar h_\infty\qquad \text{weakly-$\star$ in $L^\infty(0,L)$ as $n\to\infty$}.
\end{split}
\end{equation}
Due to the result in {\bf STEP 1}, we know that the functions $u_\infty$ and $h_\infty$ must satisfy the effective problem
\begin{equation}
\label{eq:eltauinfty-1}
\begin{cases}
-\mu\pxxxx u_\infty=h_\infty,&\quad 0<x<L,\\
\pxx u_\infty(0)= \pxx u_\infty(L)=0,&{}\\
\pxxx u_\infty(0)=\pxxx u_\infty(L)=0.&{}
\end{cases}
\end{equation}

Moreover, by \eqref{eq:conv1}  we have also that
\begin{equation}
\label{eq:hinfty}
h_\infty\in \partial\Phi'(u_\infty).
\end{equation}

By multiplying \eqref{eq:eltauinfty-1} by $u_\infty$, integrating over $(0,L)$, and recalling \eqref{eq:hinfty}, we get
\begin{equation*}
\int_0^L \mu(\pxx u_\infty)^2dx=-\int_0^L h_\infty u_\infty dx\le0,
\end{equation*}
it follows that
\begin{equation*}
\pxx u_\infty\equiv 0.
\end{equation*}
Therefore, we can conclude that \eqref{eq:aff} holds.

Using \eqref{eq:aff} in \eqref{eq:eltauinfty-1} we have also that
\begin{equation*}
h_\infty\equiv 0,
\end{equation*}
hence, due to \eqref{eq:hinfty} only one within \eqref{eq:s1},\eqref{eq:s2}, \eqref{eq:s3} can occur.

\section{Linearization Property}
\label{sec:lin}

\begin{proof}[Proof of Theorem \ref{th:lin}] 
Thanks to \eqref{ass:lin1}
\begin{equation*}
u_{0,n}\to0\quad\text{uniformly in $[0,L]$}.
\end{equation*}
Therefore using also \eqref{ass:lin3}, the assumptions of Theorem \ref{th:adhesion} are fulfilled, hence
\begin{equation}
\label{eq:lin0}
\norm{u_n}_{L^\infty((0,\infty)\times(0,L))}<1.
\end{equation}

In addition, the function 
\begin{equation*}
w_n=u_n-v_n
\end{equation*}
is a dissipative solution of 
\begin{equation}
\label{eq:w}
\begin{cases}
\rho\ptt w_n=-\mu\pxxxx w_n-\Phi'(u_n)&\quad t>0,0<x<L,\\
\pxx w_n(t,0)=\pxx w_n(t,L)=0&\quad t>0,\\
\pxxx w_n(t,0)=\pxxx w_n(t,L)=0&\quad t>0,\\
w_n(0,x)=\pt w_n(0,x)=0&\quad 0<x<L.
\end{cases}
\end{equation}
Multiplying \eqref{eq:w} by $\pt w_n$ we gain
\begin{align*}
\frac{d}{dt}\int_0^L\frac{\rho(\pt w_n)^2+\mu(\pxx w_n)^2}{2}dx=&\int_0^L\Phi'(u_n)\pt w_n dx\\
\le&\frac{1}{2\rho}\int_0^L\Phi'(u_n)^2dx+\int_0^L\frac{\rho(\pt w_n)^2}{2}dx\\
\le&\frac{1}{2\rho}\int_0^L\Phi'(u_n)^2dx+\int_0^L\frac{\rho(\pt w_n)^2+\mu(\pxx w_n)^2}{2}dx
\end{align*}
and applying the Gronwall Lemma
\begin{equation}
\label{eq:lin1}
\int_0^L\frac{\rho(\pt w_n(t,x))^2+\mu(\pxx w_n(t,x))^2}{2}dx\le \frac{1}{2\rho}\int_0^t\int_0^L e^{t-s}\Phi'(u_n(s,x))^2dsdx.
\end{equation}

Thanks to Theorem \ref{th:exist} and \cite[Theorem 5]{S}, \eqref{eq:lin0}
 there exists a function $u$ satisfying items ($i$) and ($ii$) in Definition \ref{def:sol}   such that,
passing to a subsequence,
\begin{equation}
\label{eq:convlin}
\begin{split}
&\norm{u}_{L^\infty((0,\infty)\times(0,L))}\le 1,\\
 &u_n\weak u \quad \text{in $H^1((0,T)\times(0,L))$ and in $L^2(0,T;H^2(0,L))$, for each $T\ge0$},\\ 
 &u_n \to u \quad \text{in $L^\infty((0,T)\times(0,L))$, for each $T\ge0$},\\
 &\Phi'(u_n) \weak \Phi'(u) \quad \text{in $L^p((0,T)\times(0,L))$, for each $T\ge0$ and $1\le p<\infty$}.
\end{split}
\end{equation}
Therefore $u$ is a distributional solution of 
\begin{equation}
\label{eq:ulin}
\begin{cases}
\rho\ptt u=-\mu\pxxxx u-\Phi'(u)&\quad t>0,0<x<L,\\
\pxx u(t,0)=\pxx u(t,L)=0&\quad t>0,\\
\pxxx u(t,0)=\pxxx u(t,L)=0&\quad t>0,\\
u(0,x)=\pt u(0,x)=0&\quad 0<x<L.
\end{cases}
\end{equation}
Since $u$ takes values in $[-1,1]$ and $\Phi$ is $C^2$ therein we can differentiate \eqref{eq:ulin}
and get
\begin{equation*}
\rho\pttt u=-\mu\ptxxxx u-\Phi''(u)\pt u.
\end{equation*}
Multiplying by $\ptt u$, using \ref{ass:phi} and ($ii$) in Definition \ref{def:sol} through a regularization argument we get
\begin{align*}
\frac{d}{dt}\int_0^L\frac{\rho(\ptt u)^2+\mu(\ptxx u)^2}{2}dx=&\int_0^L\Phi''(u)\pt u \ptt udx\\
\le&\frac{1}{2\rho}\int_0^L\Phi''(u)^2(\pt u)^2dx+\int_0^L\frac{\rho(\ptt u)^2}{2}dx\\
\le&c+\int_0^L\frac{\rho(\ptt u)^2+\mu(\ptxx u)^2}{2}dx.
\end{align*}
Thanks to the Gronwall Lemma 
\begin{equation*}
\int_0^L\frac{\rho(\ptt u(t,x))^2+\mu(\ptxx u(t,x))^2}{2}dx\le c(e^t-1),\qquad t\ge0.
\end{equation*}
As a consequence $u$ is an energy preserving solution of \eqref{eq:ulin} and then it must be the trivial one.
Eventually, \eqref{eq:lin1} concludes the proof.
\end{proof}

\section{{Non-uniqueness and lack of smoothness}}
\label{sec:uniqu}

This section is devoted to exploit some qualitative properties of \eqref{eq:el} through explicit analytical examples 
evidencing the lack of uniqueness and smoothness of solutions.
A key mechanism ruling these phenomena relies in the transition between the two configurations induced by the discontinuity affecting the forcing term $\Phi'$.
In particular, the first two examples show the lack of uniqueness while the last one and the numerical experiments enlighten  the occurrences of lack of smoothness.

\begin{example}\normalfont
\label{ex:2}
Let $\eps>0$ and set $\rho=\mu=1$. Consider the function
\begin{equation*}
\Phi_\eps(u)=\begin{cases}
\frac{2-\eps}{2}u^2,&\quad  \text{if $|u|\le 1,$}\\
\frac{2-\eps}{\eps}\left((1+\eps)\left(u-\frac{1}{2}\right)-\frac{u^2}{2}\right),&\quad  \text{if $1\le u\le 1+\eps,$}\\
\frac{\eps-2}{\eps}\left((1+\eps)\left(u+\frac{1}{2}\right)+\frac{u^2}{2}\right),&\quad  \text{if $-1-\eps\le u\le -1,$}\\
\frac{(2-\eps)(1+\eps)}{2},&\quad  \text{if $ |u|\ge 1+\eps.$}
\end{cases}
\end{equation*}
We have
\begin{equation*}
\Phi_\eps'(u)=\begin{cases}
(2-\eps)u,&\quad  \text{if $|u|\le 1,$}\\
\frac{2-\eps}{\eps}(1+\eps-u),&\quad  \text{if $1\le u\le 1+\eps,$}\\
\frac{\eps-2}{\eps}(1+\eps+u),&\quad  \text{if $-1-\eps\le u\le -1,$}\\
0,&\quad  \text{if $|u|\ge 1+\eps.$}
\end{cases}
\end{equation*}
The functions 
\begin{equation*}
u_\eps(t,x)=(1-\eps)\cos\left(\sqrt{2-\eps}\,t\right),\qquad v_\eps(t,x)=1+\eps
\end{equation*}
solve
\begin{align}
\label{eq:ex.2.1}
&\begin{cases}
\ptt u_\eps=-\pxxxx u_\eps-\Phi_\eps'(u_\eps),&\quad t>0,0<x<L,\\
\pxx u_\eps(t,0)=\pxx u_\eps(t,L)=0,&\quad t>0,\\
\pxxx u_\eps(t,0)=\pxxx u_\eps(t,L)=0,&\quad t>0,\\
u_\eps(0,x)=1-\eps,&\quad 0<x<L,\\
\pt u_\eps(0,x)=0,&\quad 0<x<L,
\end{cases}
\\&
\label{eq:ex.2.2}
\begin{cases}
\ptt v_\eps=-\pxxxx v_\eps-\Phi_\eps'(v_\eps),&\quad t>0,0<x<L,\\
\pxx v_\eps(t,0)=\pxx v_\eps(t,L)=0,&\quad t>0,\\
\pxxx v_\eps(t,0)=\pxxx v_\eps(t,L)=0,&\quad t>0,\\
v_\eps(0,x)=1+\eps,&\quad 0<x<L,\\
\pt v_\eps(0,x)=0,&\quad 0<x<L.
\end{cases}
\end{align}
As $\eps\to0$ we have
\begin{equation*}
u_\eps(t,x)\to u(t,x)=\cos\left(\sqrt{2}\,t\right),\qquad v_\eps(t,x)\to v(t,x)=1,
\end{equation*}
and $u$ and $v$ provides two different solutions of \eqref{eq:el} in correspondence of the initial data
\begin{equation*}
u_0(x)=1,\qquad u_1(x)=0.
\end{equation*}
The energies associated to \eqref{eq:ex.2.1} and \eqref{eq:ex.2.2} are 
\begin{align*}
E_\eps[u_\eps](t)&=\int_0^L\left(\frac{(\pt u_\eps(t,x))^2+(\pxx u_\eps(t,x))^2}{2}+ \Phi_\eps(u_\eps(t,x))\right)dx=\frac{(2-\eps)(1-\eps)^2}{2}L,\\
E_\eps[v_\eps](t)&=\int_0^L\left(\frac{(\pt v_\eps(t,x))^2+(\pxx v_\eps(t,x))^2}{2}+\Phi_\eps(v_\eps(t,x))\right)dx=\frac{(2-\eps)(1+\eps)}{2}L,
\end{align*}
respectively.
\end{example}

\begin{example}\normalfont
\label{ex:3}
For every $\eps>0$, the solutions $u_\eps$ and $v_\eps$ of the two following problems
\begin{align}
\label{eq:ex.3.1}
&\begin{cases}
\ptt u_\eps=-\pxxxx u_\eps-\Phi'(u_\eps),&\quad t>0,0<x<L,\\
\pxx u_\eps(t,0)=\pxx u_\eps(t,L)=0,&\quad t>0,\\
\pxxx u_\eps(t,0)=\pxxx u_\eps(t,L)=0,&\quad t>0,\\
u_\eps(0,x)=1+\eps,&\quad 0<x<L,\\
\pt u_\eps(0,x)=\eps,&\quad 0<x<L,
\end{cases}
\\&
\label{eq:ex.3.2}
\begin{cases}
\ptt v_\eps=-\pxxxx v_\eps-\Phi'(v_\eps),&\quad t>0,0<x<L,\\
\pxx v_\eps(t,0)=\pxx v_\eps(t,L)=0,&\quad t>0,\\
\pxxx v_\eps(t,0)=\pxxx v_\eps(t,L)=0,&\quad t>0,\\
v_\eps(0,x)=1-\eps,&\quad 0<x<L,\\
\pt v_\eps(0,x)=0,&\quad 0<x<L,
\end{cases}
\end{align}
are
\begin{equation*}
u_\eps(t,x)=\eps t+1+\eps,\qquad 
v_\eps(t,x)=(1-\eps)\cos(\sqrt{2}t).
\end{equation*}
We have
\begin{align*}
&\norm{u_\eps(0,\cdot)-v_\eps(0,\cdot)}_{L^2(0,L)}+\norm{\pt u_\eps(0,\cdot)-\pt v_\eps(0,\cdot)}_{L^2(0,L)}=3\eps\sqrt{L},\\
&\lim_{t\to\infty}u_\eps(t,x)=\infty,\qquad \limsup_{t\to\infty}v_\eps(t,x)=1-\eps.
\end{align*}
Moreover, as $\eps\to0$,
\begin{equation*}
u_\eps(t,x)\to 1,\qquad 
v_\eps(t,x)\to \cos(\sqrt{2}t).
\end{equation*}
The energies associated to \eqref{eq:ex.3.1} and \eqref{eq:ex.3.2} are 
\begin{align*}
E[u_\eps](t)&=\int_0^L\left(\frac{(\pt u_\eps(t,x))^2+(\pxx u_\eps(t,x))^2}{2}+ \Phi(u_\eps(t,x))\right)dx=\frac{\eps^2+2}{2}L,\\
E[v_\eps](t)&=\int_0^L\left(\frac{(\pt v_\eps(t,x))^2+(\pxx v_\eps(t,x))^2}{2}+\Phi(v_\eps(t,x))\right)dx=(1-\eps)^2L,
\end{align*}
respectively.
\end{example}

\begin{example}\normalfont
\label{ex:reg}
Consider the function
\begin{equation}
\label{eq:reg.1}
u(t,x)=
\begin{cases}
\sqrt{2}\sin(\sqrt{2}t),&\qquad \text{if $0\le t\le \frac{\pi}{4\sqrt{2}}$},\\
\sqrt{2} t+1-\frac{\pi}{4},&\qquad \text{if $t\ge \frac{\pi}{4\sqrt{2}}$}.
\end{cases}
\end{equation}
Clearly, $u$ solves the problem
\begin{equation*}
\begin{cases}
\ptt u=-\pxxxx u-\Phi'(u),&\quad t>0,x\in(0,L),\\
\pxx u(t,0)=\pxx u(t,L)=0,&\quad t>0,\\
\pxxx u(t,0)=\pxxx u(t,L)=0,&\quad t>0,\\
u(0,x)=0,&\quad x\in(0,L),\\
\pt u(0,x)=2,&\quad x\in(0,L),
\end{cases}
\end{equation*}
but 
\begin{equation*}
u\in C^1([0,\infty)\times[0,L])\setminus C^2([0,\infty)\times[0,L]).
\end{equation*}
Indeed
\begin{align*}
\lim_{t \to \frac{\pi}{4\sqrt{2}}^- }u\left(t,x\right)=1,\qquad & \lim_{t \to \frac{\pi}{4\sqrt{2}}^+}u\left(t,x\right)=1,\\
\lim_{t \to \frac{\pi}{4\sqrt{2}}^-}\pt u\left(t,x\right)=\sqrt{2},\qquad & \lim_{t \to \frac{\pi}{4\sqrt{2}}^+} \pt u\left(t,x\right)=\sqrt{2},\\
\lim_{t \to \frac{\pi}{4\sqrt{2}}^-}\ptt u\left(t,x\right)=-2,\qquad & \lim_{t \to\frac{\pi}{4\sqrt{2}}^+ } \ptt u\left(t,x\right)=0.
\end{align*}
The energy associated to \eqref{eq:reg.1} is
\begin{equation*}
E[u](t)=\int_0^L\left(\frac{(\pt u(t,x))^2+(\px u(t,x))^2}{2}+ \Phi(u(t,x))\right)dx=2L,
\end{equation*}
for every $t\ge0$.
\end{example}


\begin{thebibliography}{10}

\bibitem{BK}
R.~Burridge and J.~B. Keller.
\newblock Peeling, slipping and cracking---some one-dimensional free-boundary
  problems in mechanics.
\newblock {\em SIAM Rev.}, 20(1):31--61, 1978.

\bibitem{CFM}
G.~M. Coclite, G.~Fanizza, and F.~Maddalena.
\newblock Regularity and energy transfer for a nonlinear beam equation.
\newblock {\em Appl. Math. Lett.}, 115:106959, 2021.

\bibitem{CFLM}
G.~M. Coclite, G.~Florio, M.~Ligab\`o, and F.~Maddalena.
\newblock Nonlinear waves in adhesive strings.
\newblock {\em SIAM J. Appl. Math.}, 77(2):347--360, 2017.

\bibitem{CFLM-1}
G.~M. Coclite, G.~Florio, M.~Ligab\`o, and F.~Maddalena.
\newblock Adhesion and debonding in a model of elastic string.
\newblock {\em Comput. Math. Appl.}, 78(6):1897--1909, 2019.

\bibitem{G96}
P.~{G\'erard}.
\newblock {Oscillations of concentration effects in semilinear dispersive wave
  equations}.
\newblock {\em {J. Funct. Anal.}}, 141(1):60--98, 1996.

\bibitem{LL}
E.~H. Lieb and M.~Loss.
\newblock {\em Analysis}, volume~14 of {\em Graduate Studies in Mathematics}.
\newblock American Mathematical Society, Providence, RI, second edition, 2001.

\bibitem{MP}
F.~Maddalena and D.~Percivale.
\newblock Variational models for peeling problems.
\newblock {\em Interfaces Free Bound.}, 10(4):503--516, 2008.

\bibitem{MPPT}
F.~Maddalena, D.~Percivale, G.~Puglisi, and L.~Truskinovsky.
\newblock Mechanics of reversible unzipping.
\newblock {\em Contin. Mech. Thermodyn.}, 21(4):251--268, 2009.

\bibitem{MPT}
F.~Maddalena, D.~Percivale, and F.~Tomarelli.
\newblock Adhesive flexible material structures.
\newblock {\em Discrete Contin. Dyn. Syst. Ser. B}, 17(2):553--574, 2012.

\bibitem{MPT2}
F.~Maddalena, D.~Percivale, and F.~Tomarelli.
\newblock Local and non-local energies in adhesive interaction.
\newblock {\em IMA J. Appl. Math.}, 81(6):1051--1075, 2016.

\bibitem{MPT21}
F.~Maddalena, D.~Percivale, and F.~Tomarelli.
\newblock Elastic-brittle reinforcement of flexural structures, 2021.

\bibitem{N}
L.~Nirenberg.
\newblock On elliptic partial differential equations.
\newblock {\em Ann. Scuola Norm. Sup. Pisa Cl. Sci. (3)}, 13:115--162, 1959.

\bibitem{S}
J.~Simon.
\newblock Compact sets in the space {$L^p(0,T;B)$}.
\newblock {\em Ann. Mat. Pura Appl. (4)}, 146:65--96, 1987.

\bibitem{taylor}
M.~E. Taylor.
\newblock {\em Partial differential equations {I}. {B}asic theory}, volume 115
  of {\em Applied Mathematical Sciences}.
\newblock Springer, New York, second edition, 2011.

\end{thebibliography}
\end{document}